\documentclass[11pt,reqno]{amsart}

\usepackage{amsmath,amsthm,amssymb,easybmat, mathrsfs, enumerate, verbatim}
\usepackage{graphicx,lscape,chngcntr,color}
\usepackage[tmargin=1.0in,bmargin=1.0in,rmargin=1.0in,lmargin=1.0in]{geometry}
\usepackage[breaklinks=true]{hyperref}

\theoremstyle{plain}

\newtheorem{theorem}{\textrm{\textbf{Theorem}}}
\newtheorem*{utheorem}{\textrm{\textbf{Theorem}}}
\newtheorem{corollary}[theorem]{\textrm{\textbf{Corollary}}}
\newtheorem{proposition}[theorem]{\textrm{\textbf{Proposition}}}
\newtheorem{lemma}[theorem]{\textrm{\textbf{Lemma}}}

\theoremstyle{definition}

\newtheorem{remark}[theorem]{\textrm{\textbf{Remark}}}

\theoremstyle{remark}

\numberwithin{equation}{section}

\def\sgn {\mathop{\rm sgn}}

\def\C{\mathbb{C}}
\def\R{\mathbb{R}}
\def\F{\mathbb{F}}
\def\Q{\mathbb{Q}}
\def\Z{\mathbb{Z}}
\def\N{\mathbb{N}}

\newcommand{\bp}{\ensuremath{\mathbb P}}
\newcommand{\I}{\ensuremath{\widetilde{I}}}

\begin{document}
\title[Loewner positive functions, and solutions of Cauchy's functional
equations]{Loewner positive entrywise functions, and classification of
measurable solutions of Cauchy's functional equations}
\author[Dominique Guillot \and Apoorva Khare \and Bala
Rajaratnam]{Dominique Guillot \and Apoorva Khare \and Bala
Rajaratnam\\Stanford University}
\address{Department of Mathematics, Stanford University}
\date{\today}
\subjclass[2010]{39B22 (primary); 39B52, 39B42 (secondary)}
\keywords{Loewner positivity, entrywise maps, multiplicative maps,
additive maps, locally measurable, Cauchy functional equations}

\begin{abstract}
Entrywise functions preserving Loewner positivity have been studied by
many authors, most notably Schoenberg and Rudin. Following their work, it
is known that functions preserving  positivity when applied entrywise to
positive semidefinite matrices of all dimensions are necessarily analytic
with nonnegative Taylor coefficients. When the dimension is fixed, it has
been shown by Vasudeva and Horn that such functions are automatically
continuous and sufficiently differentiable. A natural refinement of the
aforementioned problem consists of characterizing functions preserving
positivity under rank constraints. In this paper, we begin this study by
characterizing entrywise functions which preserve the cone of positive
semidefinite real matrices of rank $1$ with entries in a general
interval. Classifying such functions is intimately connected to the
classical problem of solving Cauchy's functional equations, which have
non-measurable solutions. We demonstrate that under mild local
measurability assumptions, such functions are automatically smooth and
can be completely characterized. We then extend our results by
classifying functions preserving positivity on rank $1$ Hermitian complex
matrices.
\end{abstract}
\maketitle

\section{Introduction}

Given a function $f: \R \to \R$ and a matrix $A := (a_{ij})_{i,j=1}^n$,
denote by $f[A] := (f(a_{ij}))_{i,j=1}^n$ the matrix obtained by applying
$f$ to $A$ entrywise. Such entrywise functions of matrices have been
well-studied in the literature (see e.g.~\cite{Schoenberg42, Rudin59,
Herz63, Horn, Christensen_et_al78, vasudeva79, FitzHorn, fitzgerald,
Hiai2009, GKR-crit-2sided}), and have recently received renewed attention
due to their application in the regularization of high-dimensional
covariance matrices (see e.g.~\cite{bickel_levina,
Guillot_Rajaratnam2012b, Guillot_Rajaratnam2012} and the references
therein). A natural and important question consists of classifying the
functions $f$ for which $f[A]$ is positive semidefinite for every $n
\times n$ positive semidefinite matrix $A$. This question was previously
studied by Schoenberg \cite{Schoenberg42} and Rudin \cite{Rudin59}, who
showed that functions preserving Loewner positivity for all dimensions $n
\geq 2$ are automatically analytic with nonnegative Taylor coefficients.
When the dimension $n$ is fixed, it can be shown \cite{vasudeva79, Horn}
that the function $f$ has to satisfy certain continuity and
differentiability assumptions on $(0,\infty)$.

A natural refinement of this problem consists of classifying functions
preserving positivity 1) under rank constraints and 2) on arbitrary
domains in $\R$. Given $I \subset \R$, denote by $\bp_n^k(I)$ the cone of
positive semidefinite $n \times n$ matrices having rank at most $k$. As
we will show later, given $1 \leq k \leq n$, there exist non-measurable
entrywise functions mapping $\bp_n^1(I)$ into $\bp_n^k(\R)$. This is in
sharp contrast to the case where no rank constraint is imposed, where $f$
is automatically continuous on $(0,\infty)$ (see \cite[Theorem
2]{vasudeva79}). 

In this paper, we characterize entrywise functions mapping $\bp_n^1(\I)$
into $\bp_n^1(\R)$ under minimal additional hypotheses, and for arbitrary
intervals $\I \subset \R$. We demonstrate how a weak assumption of
measurability implies that every $f: \I \to \R$ mapping $\bp_n^1(\R)$
into $\bp_n^1(\R)$ is smooth on the whole interval $\I$ except maybe at
the origin. To state our main result, we introduce some notation. Given
$\alpha \in \R$, define
\begin{equation}
\phi_\alpha(x) := |x|^\alpha, \quad \psi_\alpha(x) := \sgn(x) |x|^\alpha,
\quad x \in \R \setminus \{ 0 \}, 
\end{equation}
and let $\phi_\alpha(0) = \psi_\alpha(0) := 0$. The following theorem is
the main result of the paper.

\begin{utheorem}[Main result]
Let $I \subset \I \subset \R$ be intervals containing $1$ as an interior
point. Suppose $\I \cap (0,\infty)$ is open, $\pm \sup \I \not\in \I$,
and let $n \geq 3$. Then the following are equivalent for a function $K:
\I \to \R$ which is Lebesgue measurable on $I$.
\begin{enumerate}
\item $K \not\equiv 0$ on $\I$ and $K[A] := (K(a_{ij}))_{i,j=1}^n \in
\bp_n^1(\R)$ for all $A \in \bp_n^1(\I)$.
\item $K(1) \neq 0$ and $K[A] \in \bp_n^1(\R)$ for all $A \in
\bp_n^1(\I)$.
\item Either $K$ is a positive constant, or $K$ is a positive scalar
multiple of $\phi_\alpha$ or $\psi_\alpha$ on $\I$ for some $\alpha \in
\R$.
\end{enumerate}

\noindent Moreover, the distinct maps among $\{ \phi_\alpha, \psi_\alpha
: \alpha \in \R \}$ and the constant map $K \equiv 1$ are linearly
independent on $\I$.
\end{utheorem}

Note that when $I = \I = (0,\infty)$, the rank constraint in the Main
Theorem is equivalent to the vanishing of all $2 \times 2$ minors of
$f[A]$. Thus, in that case, the functions mapping $\bp_n^1(\I)$ into
$\bp_n^1(\R)$ are precisely the measurable solutions of Cauchy's power
functional equation $K(xy) = K(x) K(y)$ on $(0,\infty)$, which have been
classified by  Sierpi\'nsky \cite{Sierpinsky}, Banach \cite{Banach},
Alexiewicz and Orlicz \cite{Orlicz}, and others. The analysis is much
more involved for a general interval $\I$. In proving the Main Theorem of
the paper, our strategy is as follows: in Section \ref{Scauchy}, we
classify the solutions of Cauchy's additive and multiplicative functional
equations 1) on a general interval $\I \subset \R$, and 2) assuming
Lebesgue measurability only on an arbitrary subinterval $I \subset \I$.
In doing so, we generalize previous work on the additive Cauchy
functional equation by Sierpi\'nsky \cite{Sierpinsky} and Banach
\cite{Banach} to this more general setting. We then prove the Main
Theorem in Section \ref{Sproof} by showing that every entrywise function
mapping $\bp_n^1(\I)$ into $\bp_n^1(\R)$ satisfies Cauchy's
multiplicative functional equation on $\I$. Although this result is
expected, its proof is intricate and requires carefully examining the
minors of various matrices. Next, in Section \ref{Scomplex}, we
demonstrate how our main result can be extended to the case where
matrices are also allowed to have complex entries. Our characterization
involves an interesting family of complex power functions
$\Psi_{\alpha,\beta} : r e^{i \theta} \mapsto r^\alpha e^{i \beta
\theta}$ for $\alpha \in \R$ and $\beta \in \Z$. We conclude the paper by
showing how our techniques can also be applied to classify the solutions
of other functional equations, under domain and measurability
constraints.\medskip

\noindent\textbf{Notation:}
Given a subset $S \subset \C$ and integers $1 \leq k \leq n \in \N$,
denote by $\bp_n^k(S)$ the set of $n \times n$ Hermitian positive
semidefinite matrices with entries in $S$ and rank at most $k$. Also
define $\bp_n(S) := \bp_n^n(S)$, and $S_+ := S \cap (0,\infty), S_- := S
\cap (-\infty,0)$. We denote the complex disc centered at $a \in \C$ and
of radius $0 < R \leq \infty$ by $D(a,R)$. We write $A \geq 0$ to denote
that $A \in \bp_n(\C)$, and write $A \geq B$ when $A - B \in \bp_n(\C)$.
We denote by $I_n$ the $n \times n$ identity matrix, and by ${\bf 0}_{n
\times n}$ and ${\bf 1}_{n \times n}$ the $n \times n$ matrices with
every entry equal to $0$ and $1$ respectively. Finally, we denote the
conjugate transpose of a vector or matrix $A$ by $A^*$.

\section{Cauchy functional equations}\label{Scauchy}

The four {\it Cauchy functional equations}
\begin{alignat}{3}
& (a)\ K(x+y) = K(x) + K(y), & \qquad & (b)\ K(xy) = K(x) K(y),
\label{Ecauchy}\\
& (c)\ K(x+y) = K(x) K(y),   & \qquad & (d)\ K(xy) = K(x) + K(y)\notag
\end{alignat}
(for all $x,y \in \R$) have been well-studied in the literature
\cite{Aczel-Dhombres,Czerwik,Kannappan}. It is not difficult to show that
the following families of nonconstant continuous functions are solutions
to these equations:
(a) linear maps $K(x) = \alpha x$;
(b) power functions $K(x) = x^\alpha$;
(c) exponential maps $K(x) = e^{\alpha x}$;
(d) logarithm maps $K(x) = \ln (x^\alpha)$ respectively.
More generally, as was shown independently by Sierpi\'nsky
\cite{Sierpinsky} and Banach \cite{Banach} (see also Alexiewicz and
Orlicz \cite{Orlicz}), the linear maps are also the only solutions of (a)
when $K$ is only Lebesgue measurable. The Cauchy multiplicative
functional equation (b) has also been studied in other related settings
such as the Loewner and Lorentz cones \cite{Wesolowski-psd, Wesolowski}. 

In order to prove our Main Theorem we begin by studying the nonconstant
solutions to Equation \eqref{Ecauchy}(a) on (1) general intervals $\I
\subset \R$ and (2) assuming Lebesgue measurability only on a subinterval
$I \subset \I$. Our main result extends previous work by Sierpi\'nsky
\cite{Sierpinsky} and Banach \cite{Banach}, where they consider the case
$I = \I = \R$.

\begin{theorem}[Additive Cauchy functional equation]\label{Tsierpinsky}
Let $I \subset \I \subset \R$ be intervals containing $0$ as an interior
point. Then the following are equivalent for $g : \I \to \R$.
\begin{enumerate}
\item $g$ is Lebesgue measurable on $I$ and additive on $\I$.
\item $g$ is continuous on $I$ and additive on $\I$.
\item $g$ is linear on $\I$ - i.e., $g(x) = \beta x = \beta \psi_1(x)$
for some $\beta \in \R$ and all $x \in \I$.
\end{enumerate}
\end{theorem}
\begin{proof}
Clearly, $(3) \implies (2) \implies (1)$. In order to show that $(1)
\implies (3)$, we first adapt the arguments in \cite{Sierpinsky} to the
interval $I$ as follows. Set $R := \min(|\inf I|, \sup I)$, and fix a
rational number $a_0 \in (0,R)$. Suppose $g(a_0) = \beta a_0$ for some
$\beta$.
One now shows easily that $g(a_0/m) = \beta (a_0/m)$ for all $m \in \N$.
We then claim that $g(a)-g(b) = \beta (a-b)$ for all $a,b \in I$ with
$a-b$ rational. Indeed, suppose $a_0 = p/q$ and $a-b = r/s > 0$ for
integers $p,q,r,s \in \N$. Now for any integer $N > \frac{R}{qs}$,
\[ g(a) - g(b) = \sum_{i=1}^{Nqr} g(a - (i-1)/(Nqs)) - g(a - i/(Nqs)) =
Nqr \cdot g \left( \frac{1}{Nqs} \right) = \beta \frac{Nqr}{Nqs} =
\beta(a-b). \]

\noindent Moreover, $g$ is clearly odd, so the sets
\[ E_\pm := \{ x \in (-R,R) : \pm (g(x) - \beta x) > 0 \} \]

\noindent satisfy: $E_- = - E_+$. Now if $E_\pm$ have positive Lebesgue
measure, then by a classical result \cite[Lemma 2]{Sierpinsky}, there
exist $e_\pm \in E_\pm$ such that $e_+ - e_- \in \Q$. This is a
contradiction since $g(e_\pm) - \beta e_\pm$ are of different signs.

We next claim that $g(x) = \beta x$ for all $x \in (-R,R)$. Suppose this
is false, and $g(a) \neq \beta a$ for some $a \in (-R,R)$. Now since
$E_\pm$ have Lebesgue measure zero, the set $G := \{ x \in (2|a|-R,R) :
g(x) = \beta x \}$ has positive measure, whence so does $G' := \{ x \in
(|a| - R, R - |a|) : g(x+a) = \beta(x+a) \}$. But since $g(a) \neq \beta
a$, hence $G' \subset E_- \coprod E_+$ and $E_\pm$ both have Lebesgue
measure zero, which is a contradiction. We conclude that $g(x) = \beta x$
for all $x \in (-R,R)$. By additivity, $g(x) = \beta x$ for all $x \in
\I$, proving (3).
\end{proof}

Using Theorem \ref{Tsierpinsky}, we can now classify the solution of
Cauchy's multiplicative functional equations to general intervals, under
our local measurability assumption. The following result plays a crucial
role in the proof of the Main Theorem.

\begin{theorem}\label{Tmult}
Let $I \subset \I \subset \R$ be intervals containing $1$ as an interior
point. Then the following are equivalent for a function $K : \I \to \R$
that is not identically zero on $\I$.
\begin{enumerate}
\item $K(1) \neq 0$, $K$ is Lebesgue measurable on $I$, and $K / K(1)$ is
multiplicative on $\I$.

\item $K(1) \neq 0$, $K$ is monotone on $I_\pm$, and $K / K(1)$ is
multiplicative on $\I$.

\item $K(1) \neq 0$, $K$ is continuous on $I \setminus \{ 0 \}$, and $K /
K(1)$ is multiplicative on $\I$.

\item $K(1) \neq 0$, $K$ is differentiable on $I \setminus \{ 0 \}$, and
$K / K(1)$ is multiplicative on $\I$.

\item Either $K$ is a nonzero constant on $\I$, or $K$ is a scalar
multiple of $\phi_\alpha$ or $\psi_\alpha$ for some $\alpha \in \R$.
\end{enumerate}
\end{theorem}

\begin{proof}
That $(5) \implies (4) \implies (3) \implies (1)$ and $(5) \implies (2)
\implies (1)$ are standard. It therefore suffices to show that $(1)
\implies (5)$. Assume that (1) holds and $K$ is nonconstant on $\I$. We
claim that $K$ does not change sign on $\I_+$. Indeed, since $1 \in \I$,
the interval $\I_+$ is closed under taking square roots.
Now $K(x)/K(1) = (K(\sqrt{x})/K(1))^2 \geq 0$ for all $x \in \I_+$. We
next claim that $K$ does not vanish on $\I_+$ if $K(1) \neq 0$. Indeed,
if $K(x) = 0$ for any $x \in \I_+$, then $K(x^{1/2^m}) = 0$ for all $m
\in \N$ by multiplicativity. Now since $x^{-1/2^m} \in \I$ for large
enough $m$, hence $K(1)^2 = K(x^{1/2^m}) K(x^{-1/2^m}) = 0$, which is
false. We conclude that $K / K(1)$ is positive on $\I_+$.

Now define $g : \ln \I_+ \to \R$ via: $g(x) := \ln (K(e^x)/K(1))$. Also
choose any compact subinterval $I_0 := [a,b]$ of $I$, with $\max(0, \inf
I) < a < 1 < b < \sup I$. We first claim that $x \mapsto K(e^x)$ is
Lebesgue measurable on $\ln I_0$, whence the restriction $g : \ln I_0 \to
\R$ is also Lebesgue measurable. To see the claim, given a Borel subset
$S$ in the image of $K \circ \exp : [\ln a, \ln b] \to \R$, the set
$K^{-1}(S) \subset [a,b]$ is Lebesgue measurable by assumption.
Therefore by the Borel regularity of Lebesgue measure, there exists a
Borel set $S'$ and a null Lebesgue set $N$ such that $K^{-1}(S) = S'
\Delta N$. But then $\ln K^{-1}(S) = (\ln S') \Delta (\ln N)$. Since $\ln
: [a,b] \to \R$ is absolutely continuous, it follows that $\ln K^{-1}(S)$
is also Lebesgue measurable, whence $K \circ \exp$ is Lebesgue measurable
on $\ln I_0$. Thus $g$ is Lebesgue measurable on $\ln I_0$.
Since $K/K(1)$ is multiplicative, it follows that $g$ is also additive on
$\ln \I_+ \supset \ln I_0 \supset \{ 0 \}$. Therefore $g$ is linear on
$\ln \I_+$ by Theorem \ref{Tsierpinsky}; say $g(x) = \alpha x$.
Reformulating, $K(x) = K(1) x^\alpha$ for all $0 < x \in \I$, for some
$\alpha \in \R$.

Next, if $0 \in \I$, then since $K$ is nonconstant on $\I$, choose $x_0
\in \I$ such that $K(x_0) \neq K(1)$. Then,
\[ \frac{K(0)}{K(1)} = \frac{K(0 \cdot x_0)}{K(1)} = \frac{K(0)}{K(1)}
\cdot \frac{K(x_0)}{K(1)}, \]

\noindent which implies that $K(0) = 0$.
Finally, if $\I \not\subset [0,\infty)$, define $\I'_- := \I \cap
(-\sqrt{\sup \I}, 0)$. Then $x^2 \in \I$ whenever $x \in \I'_-$, so we
compute: $\frac{K(x)^2}{K(1)^2} = \frac{K(x^2)}{K(1)} = |x|^{2 \alpha}$.
Therefore $K(x)/K(1) = \varepsilon(x) |x|^\alpha$, where $\varepsilon(x)
= \pm 1$. We now show that $\varepsilon$ is constant on $\I'_-$. Indeed,
if $x<y<0$ are in $\I$, then compute using that $y/x \in (0,1) \subset
\I$:
\[ \varepsilon(y) |y|^\alpha = \frac{K(y)}{K(1)} = \frac{K(x)}{K(1)}
\frac{K(y/x)}{K(1)} = \varepsilon(x) |x|^\alpha (y/x)^\alpha =
\varepsilon(x) |y|^\alpha. \]

\noindent This shows that $\varepsilon$ is constant on $\I'_-$, which in
turn implies the assertion (5) on all of $\I'_-$. We now show that (5)
holds on all of $\I_-$. Indeed, given any $x \in \I_-$ and $a_0 \in \I
\cap (1,\infty)$, there exists $m \in \N$ and $y \in \I'_-$ such that $y
a_0^m = x$. Hence by multiplicativity of $K / K(1)$ on $\I$,
\begin{equation}\label{Emult}
\frac{K(x)}{K(1)} = \frac{K(y)}{K(1)} \cdot \left( \frac{K(a_0)}{K(1)}
\right)^m = \varepsilon(y) |y|^\alpha \cdot (|a_0|^\alpha)^m =
\varepsilon(y) |x|^\alpha.
\end{equation}

\noindent This proves that $K(x) = c \phi_\alpha$ or $c \psi_\alpha$ for
some $\alpha \in \R$ and $c = K(1)$, on all of $\I$. We conclude that
$(1) \implies (5)$.
\end{proof}

As a consequence of Theorem \ref{Tmult}, we immediately obtain the
following corollary.

\begin{corollary}[Multiplicative Cauchy functional equation]\label{Cmult}
Let $I \subset \I \subset \R$ be intervals containing $1$ as an interior
point. Then the following are equivalent for a function $K : \I \to \R$
that is not identically zero on $\I$.
\begin{enumerate}
\item $K$ is multiplicative on $\I$ and Lebesgue measurable on $I$.

\item Either $K \equiv 1$ on $\I$, or $K \equiv \phi_\alpha$ or
$\psi_\alpha$ on $\I$ for some $\alpha \in \R$. \end{enumerate}
\end{corollary}
\begin{proof}
Clearly $(2) \implies (1)$. Conversely, if $K(1) = 0$, then $K \equiv 0$
on $\I$ by multiplicativity. Thus if $K \not\equiv 0$ then $K(1) = K(1)^2
\neq 0$. But then $K(1) = 1$, whence $K = K / K(1)$ is multiplicative on
$\I$. Now (1) follows by Theorem \ref{Tmult}.
\end{proof}

\begin{remark}\label{Rothers}
Theorem \ref{Tmult} also classifies all maps $K : \I \to \R$ such that $K
/ K(1)$ is multiplicative, and which are (a) Borel measurable, (b)
monotone, (c) continuous, or (d) differentiable, $C^n$ for some $n \in
\N$, or smooth on $I_\pm$ for some interval $1 \in I \subset \I$. The
reason these conditions are equivalent is that they are all satisfied by
the functions listed in (5) and imply Lebesgue measurability (1).
Moreover, we classify all multiplicative maps which are continuous or
$C^n$ on $\R$; see Corollary \ref{Csmooth}. Similar results can also be
proved under other hypotheses; see e.g.~\cite[Lemma 4.3]{Milgram}. 
\end{remark}

\section{Proof of the Main Theorem}\label{Sproof}

We now show the main result of this paper. In order to prove the linear
independence part of the result, we need the following generalization to
semigroups of the Dedekind Independence Theorem (see \cite[Chapter II,
Theorem 12]{artin_galois}). We include a short proof of this result for
convenience. 

\begin{lemma}\label{Ldedekind}
Suppose $(G,\cdot)$ is any semigroup and $\F$ any field. Let $n \geq 1$
and let $\chi_1, \dots, \chi_n : (G, \cdot) \to \F$ denote pairwise
distinct multiplicative maps that are not identically zero on $G$. Then
$\chi_1, \dots, \chi_n$ are $\F$-linearly independent.
\end{lemma}

\begin{proof}
The proof is by induction on $n$. For $n=1$ the result is clear. Suppose
it holds for $n-1 \geq 1$, and suppose $T := \sum_{j=1}^n c_j \chi_j$ is
identically zero as a function on $G$. Now choose $g_n \in G$ such that
$\chi_n(g_n) \neq \chi_1(g_n)$. Then we obtain:
\[ 0 = T(g g_n) - T(g) \chi_n(g_n) = \sum_{j=1}^n c_j (\chi_j(g)
\chi_j(g_n) - \chi_j(g) \chi_n(g_n)) = \sum_{j=1}^{n-1} c_j (\chi_j(g_n)
- \chi_n(g_n)) \chi_j(g). \]

\noindent Since $\chi_1(g_n) \neq \chi_n(g_n)$, the above sum is a
nontrivial linear combination of the (pairwise distinct) characters
$\chi_1, \dots, \chi_{n-1}$. Hence by the induction hypothesis, $c_1 =
\dots = c_{n-1} = 0$, which then implies that $c_n = 0$ as well. This
concludes the proof. 
\end{proof}

We can now prove our main result. Note that when there is no constraint
on the domain of the function $K$ in our main theorem, i.e., when $\I =
\R$, the function automatically satisfies Cauchy's multiplicative
functional equation \eqref{Ecauchy}(b), and the Main Theorem follows
immediately from Theorem \ref{Tmult}. However, as we now show, proving
that $K$ satisfies Equation \eqref{Ecauchy}(b) on $\I$ under the domain
constraint is much more involved. 

\begin{proof}[{\bf Proof of the Main Theorem}]
For ease of exposition, we write out the proof in several steps.\medskip

\noindent {\bf Step 1.}
We first show the equivalence for $n \geq 3$. It is clear that $(3)
\implies (2) \implies (1)$. We now show that $(1) \implies (2)$. Suppose
to the contrary that $K(1) = 0$. Define $u_x := (1, \dots, 1, x)^T$ for
$x \in (\sqrt{\inf \I_+}, \sqrt{\sup \I_+})$. Now examine the minor of
$K[u_x u_x^T] \in \bp_n^1(\R)$, formed by the last two rows and columns.
We conclude that $K \equiv 0$ on $(\sqrt{\inf \I_+}, \sqrt{\sup \I_+})$.
Note here that $\inf \I_+ < 1 < \sup \I_+$ by assumption.

We now claim that $K \equiv 0$ on $\I$, which shows by contradiction that
$(1) \implies (2)$. The first step is to show that $K \equiv 0$ on $(1,
(\sup \I)^{1-2^{-m}})$ for all $m \geq 1$. This was shown in the
preceding paragraph for $m=1$, and we show it for all $m$ by induction.
Thus, given $x \in (1, (\sup \I)^{1-2^{-m}})$ for $m \geq 2$,
\[ x_m := x^{2(2^{m-1}-1)/(2^m-1)} \in (1, (\sup \I)^{1 - 2^{-(m-1)}}).
\]

\noindent Now set $y_m := (\sup \I)^{1/(2^m-1)}$, and define $u_m :=
\sqrt{x_m} (1, \dots, 1, y_m)^T \in \R^n$. Then it is easily verified
that $u_m u_m^T \in \bp_n^1(\I_+)$, whence the lower rightmost $2 \times
2$ minor of $K[u_m u_m^T]$ is zero. Since $K(x_m) = 0$ by the induction
hypothesis, we conclude that $0 = K(x_m y_m) = K(x)$.
This shows that $K \equiv 0$ on $[1,\sup \I)$. A similar argument shows
that $K \equiv 0$ on $(\inf \I_+, 1]$. Thus, $K \equiv 0$ on $\I_+$. If
$0 \in \I$ then evaluating $K$ entrywise on ${\bf 1}_{1 \times 1} \oplus
{\bf 0}_{(n-1) \times (n-1)} \in \bp_n^1(\I)$ shows that $K(0) = 0$ as
well. Finally, if $\I_-$ is nonempty, then the assumptions on $\I$ imply
that $\I_- \subset -\I_+$. Thus, given $x \in \I_-$, evaluating $K$
entrywise on the matrix $\begin{pmatrix} |x| & x\\ x & |x| \end{pmatrix}
\oplus {\bf 0}_{(n-2) \times (n-2)} \in \bp_n^1(\I)$ shows that $K(x) =
0$ as well. We conclude that $K \equiv 0$ on $\I$, which shows by
contradiction that $(1) \implies (2)$.\medskip

\noindent {\bf Step 2.}
We next prove that $(2) \implies (3)$ for $n \geq 3$.
Suppose (2) holds and $K$ is nonconstant on $\I$. Define $I' :=
(\sqrt{\inf \I_+}, \sqrt{\sup \I_+}) \subset \I$. Then given $x,y \in
I'$, it is clear that $uu^T \in \bp_n^1(\I)$, where $u := (1, 1, \dots,
1, x, y)^T \in (I')^n$. Considering the $2 \times 2$ minor corresponding
to the $n-2$ and $n$th rows and the last two columns shows that $K/K(1)$
is multiplicative on $I'$. Since $K/K(1)$ is also Lebesgue measurable on
$I \cap I'$, Theorem \ref{Tmult} implies that $K$ satisfies (3) on $I'$.
Now set $u := (1, 1, \dots, 1, x)^T$ and $A := u u^T$, with $x \in I'$.
Applying (2) to $A$, we conclude that (3) holds on $\I_+$.

It remains to show that (3) holds even when $\I \not\subset (0,\infty)$.
First if $\I = [0,\infty)$ then choose $a>0$ such that $K(a) \neq K(0)$.
Applying $K$ entrywise to the matrix $a {\bf 1}_{1 \times 1} \oplus {\bf
0}_{(n-1) \times (n-1)} \in \bp_n^1(\I)$ yields $K(0)=0$, which proves
(3). Finally, suppose $\I_-$ is nonempty. There are then two cases:
\begin{enumerate}
\item $K(0) \neq 0$. We then claim that $K$ is constant on $\I$. Indeed,
choose $a \in \I_-$, define $u := \sqrt{|a|} (1,-1, 0, \dots, 0)^T \in
\R^n$, and $A := u u^T \in \bp_n^1(\I)$. Now the minor corresponding to
the first and third rows and columns of $K[A]$ is zero, whence $K(0) =
K(|a|)$. Considering the minor corresponding to the first and second
rows, and second and third columns, shows that $K(a) = K(|a|) = K(0)$.
This implies that $K$ is constant on $(\inf \I, |\inf \I|)$. Repeating
the same argument with $u := (\sqrt{a}, 0, \dots, 0)$ for all $0 < a \in
\I$ shows that $K$ is constant on $\I$.

\item $K(0) = 0$. Since $K$ is nonconstant, it follows by the above
analysis that (3) holds on $\I_+$. Thus for $a \in \I_-$, $K(a) = \pm
K(|a|) = \pm K(1) |a|^\alpha$ for some $\alpha \in \R$. In other words,
there exists $\varepsilon : \I_- \to \{ \pm 1 \}$ such that $K(a) =
K(|a|) \varepsilon(a)$ for all $a \in \I_-$. It remains to show that
$\varepsilon$ is constant. Indeed, if $x,y \in \I_-$ such that
$-\sqrt{|\inf \I|} < x<y<0$, then setting $u := (1, \dots, 1, y/x, x)^T$,
it follows that $K[u u^T] \in \bp_n^1(\R)$. Since the minor corresponding
to the $n-2$ and $n$th rows, and the $(n-2)$ and $(n-1)$st columns is
zero, we compute:
\[ \varepsilon(y) |y|^\alpha = \frac{K(x) K(y/x)}{K(1)^2} =
\varepsilon(x) |x|^\alpha |y/x|^\alpha, \]

\noindent from which it follows that $\varepsilon(y) = \varepsilon(x)$
whenever $x,y \in \I_-$ and $-\sqrt{|\inf \I|} < x < y < 0$. Finally if
$\inf \I \leq -1$, then we show that $\varepsilon$ is constant on all of
$\I_-$, which proves (3) on all of $\I$. Choose $y,a,x \in \I$ such that
\[ \inf \I \leq y \leq -\sqrt{ |\inf \I|} < a < 0 < |y| < x < \sup \I, \]

\noindent and define $u := \sqrt{x}(a/x, y/x, 1, \dots, 1)^T \in \R^n$.
One then verifies that $u u^T \in \bp_n^1(\I)$; now since the minor of
$K[u u^T]$ corresponding to the first and third columns and the first two
rows is zero, we obtain:
\[ \varepsilon(a) |a|^\alpha \cdot (ay/x)^\alpha = \varepsilon(y)
|y|^\alpha \cdot (a^2/x)^\alpha, \]

\noindent which shows that $\varepsilon(y) = \varepsilon(a)$. Therefore
$\varepsilon$ is constant on all of $\I_-$, as desired.
\end{enumerate}

To conclude the proof, the linear independence claim is immediate from
Lemma \ref{Ldedekind}, which we apply to the semigroup $G :=
(\max(\inf(\I),-1),1) \subset \I$.
\end{proof}

\begin{remark}
The assumptions on $\I$ in the Main Theorem are necessary in order to
obtain the above characterizations. For instance, if $S := \I \cap
(-\infty, -\sup \I]$ is nonempty, then every map $: S \to \R$ can be
extended to a map $K : \I \to \R$ preserving positivity on $\bp_n^1(\I)$,
since no element of $S$ can occur in any matrix in $\bp_n^1(\I)$. Even if
$\I \subset (0,\infty)$, there can exist other solutions if $\I_+$ is not
open. For instance, if $\sup \I$ or $\inf \I$ belongs to $\I$, then the
Kronecker delta functions $\delta_{x,\sup \I}$ or $\delta_{x,\inf \I}$
preserve Loewner positivity on $\bp_n^1(\I)$.
\end{remark}

The main theorem of the paper characterizes functions mapping
$\bp_n^1(\I)$ into itself when $n \geq 3$. It is natural to ask if the
same result holds when $n=2$. We now show that this is not
the case. 

\begin{proposition}
Let $I \subset \I \subset \R$ be intervals containing $1$ as an interior
point. Suppose $\I \cap (0,\infty)$ is open, and $\pm \sup \I \not\in
\I$. Then the following are equivalent for a function $K: \I \to \R$
which is Lebesgue measurable on $I$ and not identically $0$ on $\I$:
\begin{enumerate}
\item $K[-]$ maps $\bp_2^1(\I)$ to $\bp_2^1(\R)$.
\item There exists a function $\varepsilon : \I_- = \I \cap (-\infty,0)
\to \{ \pm 1 \}$ such that:
\begin{itemize}
\item $\varepsilon$ is Lebesgue measurable when restricted to $I \cap
\I_-$;
\item Either $K$ is a positive constant, or $K$ is a positive scalar
multiple of $\phi_\alpha$ or $\psi_\alpha$ on $\I \cap [0,\infty)$ for
some $\alpha \in \R$; and
\item $K(x) \equiv \varepsilon(x) K(|x|)$ for all $x \in \I_-$.
\end{itemize}
\end{enumerate}
\end{proposition}

\begin{proof}
First observe that for all $\varepsilon : \I_- \to \{ \pm 1 \}$ and $A
\in \bp_2^1(\I)$, $K[A] \in \bp_2^1(\R)$ if and only if $(K \cdot
\varepsilon)[A] \in \bp_2^1(\R)$, by the Schur product theorem. 
Thus, every $K[-]$ satisfying (2) preserves positivity on $\bp_2^1(\I)$.
Conversely, we compute:
\[ a,b,ab \in \I_+ \quad \implies \quad \begin{pmatrix} a & \sqrt{ab}\\
\sqrt{ab} & b \end{pmatrix}, \ \begin{pmatrix} ab & \sqrt{ab}\\ \sqrt{ab}
& 1 \end{pmatrix} \quad \in \bp_n^1(\I), \]

\noindent from which it follows using the hypotheses that
\[ K(a) K(b) = K(\sqrt{ab})^2 = K(ab) K(1) = K(ab). \]

\noindent In other words, $K$ is multiplicative on $\I_+$.
Now note that if $K(1) = 0$ then $K \equiv 0$ on $\I_+$ by
multiplicativity; moreover, applying $K$ entrywise to the matrices
$\begin{pmatrix} 1 & 0\\ 0 & 0 \end{pmatrix}$ and $\begin{pmatrix} |x| &
x\\ x & |x| \end{pmatrix}$ for $x \in \I_-$ shows that $K \equiv 0$ on
$\I$. We conclude by the hypotheses that $K(1) \neq 0$. Now applying
Theorem \ref{Tmult} to $\I_+$, we conclude that on $\I \cap [0,\infty)$,
$K/K(1)$ equals either $\phi_\alpha \equiv \psi_\alpha$ or a constant.

Next, the only matrices in $\bp_2^1(\I)$ with a zero entry are of the
form ${\bf 0}_{2 \times 2}$, $\begin{pmatrix} a & 0\\ 0 & 0
\end{pmatrix}$, or $\begin{pmatrix} 0 & 0\\ 0 & a \end{pmatrix}$, with $a
\in \I_+$. Considering the two cases $K(0) = 0$ and $K(0) \neq 0$, one
verifies that the result holds if $\I \subset [0,\infty)$.
Finally, applying $K$ entrywise to the matrix $\begin{pmatrix} |x| & x\\
x & |x| \end{pmatrix} \in \bp_2^1(\I)$ for $x \in \I_-$ yields: $K(x) =
\pm K(|x|)$. In other words, $K(x) = \varepsilon(x) K(|x|)$ for some
$\varepsilon : \I_- \to \{ \pm 1 \}$. Moreover, $\varepsilon(x) = K(|x|)
/ K(x)$ is Lebesgue measurable on $I \cap \I_-$, which concludes the
proof.
\end{proof}

\section{Preserving positivity on complex rank one
matrices}\label{Scomplex}

It is natural to ask if the Main Theorem in this paper has an analogue
for matrices with complex entries.  We now provide a positive answer to
this question. Although we are interested mainly in matrices with entries
in $D(0,R)$ for $R>1$, we will prove a characterization result for more
general regions $G \subset \C$, which include sets such as $G =
\overline{D(0,1)} \cup (-1-\epsilon,1+\epsilon)$, $G = S^1 \cup
(-1-\epsilon, 1+\epsilon)$, and $G = (-R,R) \cup (D(0,R) \setminus
D(0,r))$ for $0 \leq r \leq 1 < R$ and $\epsilon > 0$. In order to state
this result, we first introduce a family of {\it complex power functions}
$\Psi_{\alpha, \beta}$ for $\alpha \in \R$ and $\beta \in \Z$ by
\begin{equation}
\Psi_{\alpha, \beta}(re^{i\theta}) := r^\alpha e^{i \beta \theta}, \quad
r > 0,\ \theta \in (-\pi, \pi], \qquad \Psi_{\alpha, \beta}(0) := 0. 
\end{equation}

\noindent Then $\Psi_{n,n}(z) = z^n$ and $\Psi_{n,-n}(z) =
\overline{z}^n$ for $n \in \Z$ and $z \neq 0$. Moreover,
$\Psi_{\alpha,\beta}$ restricted to $\R$ equals $\phi_\alpha$ if $\beta$
is even and $\psi_\alpha$ if $\beta$ is odd. Additionally,
$\Psi_{\alpha,\beta}$ is continuous on $\C^\times$ and multiplicative on
$\C$ for all $\alpha \in \R$ and $\beta \in \Z$.

In a forthcoming paper \cite{GKR-complex} we explore which of the maps
$\Psi_{\alpha,\beta}$ preserve Loewner positivity when applied entrywise
to Hermitian positive semidefinite matrices of a fixed order. The
following result provides an answer when additional rank constraints are
imposed.

\begin{theorem}\label{Tcomplex}
Suppose $n \geq 3$ and $G \subset \C$ satisfies:
\begin{itemize}
\item For all $z \in S^1$, the set $I_z := \{ a \in (0,\infty) : az \in G
\}$ is an interval containing $1$, but not its supremum if $\sup I_z >
1$.

\item $G$ is closed under the conjugation and modulus maps ($z \mapsto
\overline{z}, |z|$).

\item $\I := G \cap \R$ is an interval with $1$ as an interior point,
such that $\pm \sup \I \notin \I$.
\end{itemize}

\noindent Suppose $K : G \to \C$ is Lebesgue measurable on a sub-interval
$I \subset \I$ containing $1$ as an interior point, and either Baire
measurable or universally measurable \cite[Section 2]{Rosendal} when
restricted to (the topological group) $S^1$. Then the following are
equivalent for a function $K : G \to \C$:
\begin{enumerate}
\item $K \not\equiv 0$ on $G$, and $K[-]$ maps $\bp_n^1(G)$ to
$\bp_n^1(\C)$.

\item $K(1) > 0$, and $K/K(1) : G \to \C$ is multiplicative and
conjugation-equivariant.

\item $K(1) > 0$, and either $K \equiv K(1)$ on $G$ or there exist
$\alpha \in \R$ and $\beta \in \Z$ such that $K \equiv K(1) \cdot
\Psi_{\alpha,\beta}$ on $G$.
\end{enumerate}

\noindent Moreover, the maps $\{ \Psi_{\alpha,\beta} : \alpha \in \R,
\beta \in \Z \} \cup \{ K \equiv 1 \}$ are linearly independent on
$D(0,r)$ for any $0 < r \leq \infty$.
\end{theorem}

Note that $G$ needs to be closed under conjugation in Theorem
\ref{Tcomplex} because if $z \in G$ but $\overline{z} \notin G$, then $z$
can never occur as an entry of a matrix in $\bp_n(G)$. Similarly, if $z
\in G \setminus D(0, \sup(G \cap \R))$, then $z$ can never occur as an
entry of a matrix in $\bp_n(G)$.

\begin{proof}[{\bf Proof of Theorem \ref{Tcomplex}}]
We prove a cyclic chain of implications. That $(3) \implies (1)$ is
easily verified.
We next show that $(2) \implies (3)$. First note that since $K/ K(1)$ is
multiplicative and conjugation-equivariant on $S^1$, it is a group
endomorphism of $S^1$ into itself. Moreover, $K$ is Baire/universally
measurable on $S^1 \subset \C$. It follows by results by Banach-Pettis
(see \cite{Pettis} or \cite[Theorem 2.2]{Rosendal}) and by Steinhaus-Weil
(see \cite[Corollary 2.4]{Rosendal}) that $K|_{S^1}$ is continuous.
Therefore there exists $\beta \in \Z$ such that $K(z) = K(1) z^\beta$ for
$z \in S^1$. There are now two cases to consider. First suppose $K$ is
constant on $G \cap \R$. If $\beta \neq 0$, then
\[ K(0) = K(0 \cdot \exp(i \pi/\beta)) = K(0) \exp(i \pi/\beta)^\beta =
-K(0), \]

\noindent which implies that $K(0) = 0 = K(1)$. This contradicts (2), so
$\beta = 0$. Therefore by multiplicativity and the assumptions on $G$,
\[ K(z) = K(|z|) K(z/|z|) / K(1) = K(1) (z/|z|)^0 = K(1), \qquad \forall
z \in G \setminus \{ 0 \}. \]

\noindent Therefore $K$ is constant on $G$, proving (3). The other case
is if $K$ is nonconstant on $G \cap \R$. Then by the Main Theorem, $K
\equiv K(1) \phi_\alpha \equiv K(1) \psi_\alpha$ on $G \cap [0,\infty)$
for some $\alpha \in \R$. Now compute for $z \in G \setminus \{ 0 \}$:
\[ K(z) = K(|z| \cdot z/|z|) = \frac{1}{K(1)} K(|z|) K(z/|z|) = K(1)
|z|^\alpha (z/|z|)^\beta = K(1) \Psi_{\alpha,\beta}(z), \]

\noindent which shows (3).

Finally, we show that $(1) \implies (2)$. As the proof is intricate, we
divide it into four steps for ease of exposition.\medskip

\noindent {\bf Step 1.}
We first claim that if $K(0) \neq 0$, then $K \equiv K(0)$ on $G$ (from
which (2) follows).
Indeed, if $K(0) \neq 0$, then for all $0 < a \in G$, we have that $K(a)
= K(0)$ by considering the minor formed by the first two rows and columns
of $K[a {\bf 1}_{1 \times 1} \oplus {\bf 0}_{(n-1) \times (n-1)}] \in
\bp_n^1(\C)$. Now given $z \in G \setminus \{ 0 \}$, define $u :=
|z|^{-1/2} (z, |z|, 0, \dots, 0)^T$; then $u u^* \in \bp_n^1(G)$. Now the
vanishing of the minor formed by the first two columns and the first and
third rows of $K[u u^*] \in \bp_n^1(\C)$ implies that $K \equiv K(0)$ on
$G$.\medskip

\noindent {\bf Step 2.}
Given the previous step, we will assume that $K(0) = 0$ for the remainder
of the proof. We next claim that $K(1) > 0$, and $K/K(1)$ is
conjugation-equivariant on $G$ and multiplicative on $G \cap \R$. Indeed,
note that $K(1) \geq 0$ since $K[{\bf 1}_{n \times n}] \in \bp_n^1(\C)$.
If $K(1) = 0$ then $K(G \cap [0,\infty)) = 0$ by the Main Theorem. Now
given $z \in G \setminus \{ 0 \}$, define the vector $u := \sqrt{|z|}(1,
\dots, 1, \overline{z} / |z|)^T \in \C^n$. Applying $K$ entrywise to the
matrix $u u^* \in \bp_n^1(G)$, we conclude that $K/K(1)$ is
conjugation-equivariant on $G$, and therefore $K \equiv 0$ on $G$. Since
$K \not\equiv 0$ by hypothesis, it follows that $K(1) > 0$. Now use the
Main Theorem to infer that $K/K(1)$ is multiplicative on $G \cap
\R$.\medskip

\noindent {\bf Step 3.}
Next, we show that $K/K(1)$ is multiplicative on $S^1 \subset G$. Indeed,
given $z,z' \in S^1$, define $u := (z, \overline{z'}, 1, \dots, 1)^T \in
\C^n$; then $u u^* \in \bp_n^1(G)$, so $K[u u^*] \in \bp_n^1(\C)$. We
conclude from (1) that $K/K(1) : S^1 \to S^1$, by considering the
vanishing of the minor formed by the first and third rows and columns of
$K[u u^*] \in \bp_n^1(\C)$. Now consider the minor formed by the first
two columns and the first and third rows. The vanishing of this minor
implies that
\[ K(1) K(z') = K(\overline{z}) K(zz') = \overline{K(z)} K(zz') =
K(z)^{-1} K(zz'). \]

\noindent It follows that $K/K(1)$ is multiplicative on $S^1 \subset
G$.\medskip

\noindent {\bf Step 4.}
We now claim that $K(z) = K(|z|) K(z/|z|) / K(1)$ for all $z \in G
\setminus \{ 0 \}$. The claim would imply that $(1) \implies (2)$ (recall
that $K(0) = 0$), because if $z,z',zz' \in G \setminus \{ 0 \}$, then by
the hypotheses and the conclusions of the previous two steps,
\[ \frac{K(zz')}{K(1)} = \frac{ K(|zz'|) K(zz' / |zz'|) }{K(1)^2} =
\frac{K(|z|) K(|z'|) K(z/|z|) K(z'/|z'|)}{K(1)^4} = \frac{K(z)}{K(1)}
\cdot \frac{K(z')}{K(1)}. \]

\noindent Thus it suffices to prove the claim on all of $G \setminus \{ 0
\}$. By the previous step, the claim holds on $S^1$ and on $G \cap \R$.
Now suppose $z \in S^1$ and $0 < x < \sqrt{\sup I_z}$. Define $u := (x,
z, 1, \dots, 1)^T$; then $u u^* \in \bp_n^1(G)$. Now consider the minor
formed by the first and third columns, and second and third rows, of $K[u
u^*] \in \bp_n^1(\C)$. The vanishing of this minor yields:
\begin{equation}\label{Emult_sq_sup}
 \frac{K(xz)}{K(1)} = \frac{K(x)}{K(1)} \cdot \frac{K(z)}{K(1)}, \qquad
\forall z \in S^1, \ 0 < x < \sqrt{\sup I_z}. 
\end{equation}

It remains to show that Equation \eqref{Emult_sq_sup} also holds for $z
\in S^1$ and $x \in [\sqrt{\sup I_z}, \sup I_z)$, assuming that $\sup I_z
> 1$. This is proved similarly to Step $1$ in the proof of the Main
Theorem. Namely, given $x \in (1, (\sup I_z)^{1-2^{-m}})$ for $m \in \N$,
we claim that $K(xz) = K(x) K(z)/K(1)$. The proof is by induction on $m$;
the $m=1$ case was shown in the previous paragraph. Now suppose $m>1$;
then
\[ x_m := x^{2(2^{m-1}-1)/(2^m-1)} \in (1, (\sup I_z)^{1-2^{-(m-1)}}). \]

\noindent Now set $y_m := (\sup I_z)^{1/(2^m-1)}$, and define $u_m :=
\sqrt{x_m} (y_m,z,1, \dots, 1)^T \in \R^n$. Then it is easily verified $x
= x_m y_m$ and $u_m u_m^* \in \bp_n^1(I_z)$. Therefore the minor formed
by the first and third columns, and second and third rows of $K[u_m
u_m^*]$ vanishes. This yields:
\[ K(x_m y_m z) = \frac{K(x_m y_m) K(x_m z)}{K(x_m)} = \frac{K(y_m) K(x_m
z)}{K(1)}, \]

\noindent since $K/K(1)$ is multiplicative on $G \cap \R$ by Step 2. Now
$K(x_m z) = K(x_m) K(z)/K(1)$ by the induction hypothesis, since $x_m <
(\sup I_z)^{1 - 2^{-(m-1)}}$. Therefore,
\[ K(xz) = K(x_m y_m z) = \frac{K(y_m) K(x_m z)}{K(1)} = \frac{K(y_m)
K(x_m) K(z)}{K(1)^2} = K(x_m y_m) \cdot \frac{K(z)}{K(1)} =
\frac{K(x)K(z)}{K(1)}, \]

\noindent which proves the claim, and with it, the equivalence of the
three assertions.

Finally, note that for any $0 < r \leq \infty$, the set $G' := D(0,
\min(1,r))$ is a semigroup under multiplication, and the maps $\{
\Psi_{\alpha,\beta} : \alpha \in \R, \beta \in \Z \} \cup \{ K \equiv 1
\}$ are (not necessarily $\C^\times$-valued) pairwise distinct characters
of $G'$. Therefore by Lemma \ref{Ldedekind} they are linearly independent
on $G'$, and hence on $D(0,r)$.
\end{proof}

\begin{remark} 
Note as in Remark \ref{Rothers} that Theorem \ref{Tcomplex} also
classifies the Borel/Haar-measurable maps which preserve Loewner
positivity. On the other hand, if we do not make any measurability
assumptions about $K|_{S^1}$ in Theorem \ref{Tcomplex}, there exist
non-measurable solutions satisfying Theorem \ref{Tcomplex}(1). For
instance, consider any Hamel basis $\mathcal{B} := \{ x_\gamma \}$ of
$\R$ over $\Q$ containing $1$ and contained in $(0,2)$, and let
$\mathcal{F}$ denote the set of functions $f : \mathcal{B} \to \R$ such
that $f(1) = 0$. Now given $f \in \mathcal{F}$, define the function $K_f
: S^1 \to S^1$ as follows: write $x \in \R$ as a finite sum $\sum_\gamma
c_\gamma x_\gamma$, and define
\[ K_f : \exp(i \pi x) \mapsto \exp(i \pi \sum_\gamma c_\gamma
f(x_\gamma)), \qquad \forall x \in \R.\]

\noindent Note that $K_f$ is well-defined and multiplicative on $S^1$,
and hence preserves $\bp_n^1(\C)$ for all $n \in \N$ and $f \in
\mathcal{F}$. However, since $f$ is allowed to vary over all of
$\mathcal{F}$, the function $K_f$ is not necessarily Haar measurable.
\end{remark}

\section{Measurable solutions of Cauchy functional
equations}\label{Sother}

We conclude this paper by using our methods to complete the
classification of functions satisfying the four Cauchy functional
equations \eqref{Ecauchy}, under local Borel/Lebesgue measurability
assumptions and on general intervals. To begin,  it is natural to ask for
characterizations of the multiplicative functions $K : \I \to \R$ that
are also $C^n$ or smooth. The following result is an immediate
consequence of Theorem \ref{Tmult}. 

\begin{corollary}\label{Csmooth}
Suppose $\I \subset \R$ is an interval containing $1$ as an interior
point. Given $K : \I \to \R$ and an integer $n \geq 0$, the following are
equivalent: 
\begin{enumerate}
\item $K$ is multiplicative on $\I$ and $n$ times differentiable on $I$.

\item $K$ is multiplicative on $\I$ and $C^n$ on $I$.

\item Either $K \equiv 0$ or $K \equiv 1$ on $\I$, or $K(x) = x^\alpha$
for some {\it integer} $\alpha \in (0,n]$, or $K = \phi_\alpha$ or
$\psi_\alpha$ for some $\alpha > n$.
\end{enumerate}
\end{corollary}

We now show that the other two Cauchy functional equations
\eqref{Ecauchy} can be solved using the aforementioned classifications of
all additive and multiplicative measurable maps.

\begin{theorem}\label{Taddmult}
Let $I \subset \I \subset \R$ be intervals containing $0$ as an interior
point. Then the following are equivalent for $K : \I \to \R$.
\begin{enumerate}
\item $K$ is Lebesgue measurable on $I$ and satisfies: $K(x+y) =
K(x)K(y)$ whenever $x,y,x+y \in \I$.

\item $K$ is continuous on $I$ and satisfies: $K(x+y) = K(x)K(y)$
whenever $x,y,x+y \in \I$.

\item Either $K \equiv 0$ on $\I$, or $K(x) = \exp(\beta x)$ for some
$\beta \in \R$ and all $x \in \I$.
\end{enumerate}

Suppose instead that $I \subset \I \subset \R$ are intervals containing
$1$ as an interior point. Then the following are equivalent for $K : \I
\to \R$.
\begin{enumerate}
\item $K$ is Lebesgue measurable on $I$ and satisfies: $K(xy) = K(x) +
K(y)$ whenever $x,y,xy \in \I$.

\item $K$ is continuous on $I$ and satisfies: $K(xy) = K(x) + K(y)$
whenever $x,y,xy \in \I$.

\item Either $0 \in \I$ and $K \equiv 0$ on $\I$, or $0 \notin \I$ and
$K(x) = \beta \ln(x)$ for some $\beta \in \R$ and all $x \in \I$.
\end{enumerate}
\end{theorem}

\noindent As in Theorem \ref{Tmult}, one can replace the continuity
assumption in either condition (2) by other constraints, such as $K$
being Borel measurable, monotone, differentiable, $C^n$ for some $n$, or
smooth on $I$.

\begin{proof}
\noindent $\boldsymbol{K(x + y) = K(x) K(y):}$
For the first set of equivalences, clearly $(3) \implies (2) \implies
(1)$. We now assume (1) and first show that if $K(x) = 0$ for some $x \in
\I$, then $K \equiv 0$ on $\I$. Indeed, if $K(x) = 0$, then $K(x/n)^n =
K(x) = 0$ for all $n \in \N$. Now if $n$ is large enough, then $\pm x/n
\in \I$, whence $K(0) = K(x/n) K(-x/n) = 0$. But then $K(x) = K(x+0) =
K(x) K(0) = 0$ for all $x \in \I$, and $K \equiv 0$.

Now assume that $K \not\equiv 0$ on $\I$; then $K$ never vanishes on
$\I$. Moreover, given $x \in \I$, $K(x) = K(x/2)^2 \geq 0$, so it must be
positive. Then $g(x) := \ln K(x) : \I \to \R$ is additive on $\I$ and
Lebesgue measurable on $I$. Hence by Theorem \ref{Tsierpinsky}, $g(x) =
\beta x$ for some $\beta \in \R$, whence $K(x) = e^{g(x)} = e^{\beta x}$
for some $\beta \in \R$. This shows (3) as desired.\medskip

\noindent $\boldsymbol{K(xy) = K(x) + K(y):}$
For the second set of equivalences, once again $(3) \implies (2) \implies
(1)$. Now if $K$ satisfies (1) and $0 \in \I$, then for all $x \in \I$,
\[ K(0) = K(x \cdot 0) = K(x) + K(0) \quad \implies \quad K(x) \equiv 0\
\forall x \in \I. \]

\noindent Otherwise suppose $0 \notin \I$; then $K_1(x) := e^{K(x)} : \I
\to \R$ is multiplicative and positive on $\I$ and Lebesgue measurable on
$I$. Hence by Corollary \ref{Cmult}, $K_1 \equiv 1$ (whence $K \equiv
0$) or $K_1(x) = x^\beta$ is positive on $I$, in which case $K(x) = \beta
\ln(x)$ for $x \in \I \subset (0,\infty)$. This shows (3) and concludes
the proof.
\end{proof}


\subsection*{Concluding remarks}

The main result of this paper characterizes functions $K$ mapping
$\bp_n^1(\I)$ into $\bp_n^1(\R)$ under weak measurability assumptions. A
natural question that now arises is to classify entrywise functions
mapping $\bp_n^l(\I)$ into $\bp_n^k(\R)$ for $1 \leq k,l \leq n$ under
suitable assumptions. These maps will be explored in detail in future
work \cite{GKR-lowrank}. 
%
%

\bibliographystyle{plain}
\bibliography{biblio}

\begin{thebibliography}{10}

\bibitem{Aczel-Dhombres}
J.~Acz{\'e}l and J.~Dhombres.
\newblock {\em Functional equations in several variables}, volume~31 of {\em
  Encyclopedia of Mathematics and its Applications}.
\newblock Cambridge University Press, Cambridge, 1989.
\newblock With applications to mathematics, information theory and to the
  natural and social sciences.

\bibitem{Orlicz}
A.~Alexiewicz and W.~Orlicz.
\newblock Remarque sur l'\'equation fonctionelle {$f(x+y)=f(x)+f(y)$}.
\newblock {\em Fund. Math.}, 33:314--315, 1945.

\bibitem{artin_galois}
Emil Artin.
\newblock {\em Galois theory}.
\newblock Fifth reprinting. Notre Dame Mathematical Lectures, No. 2. University
  of Notre Dame Press, South Bend, Ind., 1959.

\bibitem{Banach}
Stefan Banach.
\newblock Sur l'\'equation fonctionelle $f(x+y)=f(x)+f(y)$.
\newblock {\em Fund. Math.}, 1:123--124, 1920.

\bibitem{bickel_levina}
Peter~J. Bickel and Elizaveta Levina.
\newblock Covariance regularization by thresholding.
\newblock {\em Ann. Statist.}, 36(6):2577--2604, 2008.

\bibitem{Wesolowski-psd}
Konstancja Bobecka and Jacek Weso{\l}owski.
\newblock Multiplicative {C}auchy functional equation in the cone of
  positive-definite symmetric matrices.
\newblock {\em Ann. Polon. Math.}, 82(1):1--7, 2003.

\bibitem{Christensen_et_al78}
Jens Peter~Reus Christensen and Paul Ressel.
\newblock Functions operating on positive definite matrices and a theorem of
  {S}choenberg.
\newblock {\em Trans. Amer. Math. Soc.}, 243:89--95, 1978.

\bibitem{Czerwik}
Stefan Czerwik.
\newblock {\em Functional equations and inequalities in several variables}.
\newblock World Scientific Publishing Co. Inc., River Edge, NJ, 2002.

\bibitem{FitzHorn}
Carl~H. Fitz{G}erald and Roger~A. Horn.
\newblock On fractional {H}adamard powers of positive definite matrices.
\newblock {\em J. Math. Anal. Appl.}, 61:633--642, 1977.

\bibitem{fitzgerald}
Carl~H. Fitz{G}erald, Charles~A. Micchelli, and Allan Pinkus.
\newblock Functions that preserve families of positive semidefinite matrices.
\newblock {\em Linear Algebra Appl.}, 221:83--102, 1995.

\bibitem{GKR-crit-2sided}
Dominique Guillot, Apoorva Khare, and Bala Rajaratnam.
\newblock Complete characterization of {H}adamard powers preserving {L}oewner
  positivity, monotonicity, and convexity.
\newblock {\em \em Technical Report, Department of Mathematics, Stanford
  University}, in submission (arXiv: 1311.1581), 2013.

\bibitem{GKR-complex}
Dominique Guillot, Apoorva Khare, and Bala Rajaratnam.
\newblock On fractional {H}adamard powers of positive block matrices.
\newblock {\em \em Technical Report, Department of Mathematics, Stanford
  University}, in submission (arXiv: 1404.6839), 2014.

\bibitem{GKR-lowrank}
Dominique Guillot, Apoorva Khare, and Bala Rajaratnam.
\newblock Preserving positivity for rank-constrained matrices.
\newblock {I}n progress, 2014.

\bibitem{Guillot_Rajaratnam2012b}
Dominique Guillot and Bala Rajaratnam.
\newblock Functions preserving positive definiteness for sparse matrices.
\newblock {\em Trans. Amer. Math. Soc.}, in print; (arXiv: 1210.3894), 2012.

\bibitem{Guillot_Rajaratnam2012}
Dominique Guillot and Bala Rajaratnam.
\newblock Retaining positive definiteness in thresholded matrices.
\newblock {\em Linear Algebra Appl.}, 436(11):4143--4160, 2012.

\bibitem{Herz63}
Carl~S. Herz.
\newblock Fonctions op\'erant sur les fonctions d\'efinies-positives.
\newblock {\em Ann. Inst. Fourier (Grenoble)}, 13:161--180, 1963.

\bibitem{Hiai2009}
Fumio Hiai.
\newblock Monotonicity for entrywise functions of matrices.
\newblock {\em Linear Algebra Appl.}, 431(8):1125--1146, 2009.

\bibitem{Horn}
Roger~A. Horn.
\newblock The theory of infinitely divisible matrices and kernels.
\newblock {\em Trans. Amer. Math. Soc.}, 136:269--286, 1969.

\bibitem{Kannappan}
Pl. Kannappan.
\newblock {\em Functional equations and inequalities with applications}.
\newblock Springer Monographs in Mathematics. Springer, New York, 2009.

\bibitem{Milgram}
Arthur~N. Milgram.
\newblock Multiplicative semigroups of continuous functions.
\newblock {\em Duke Math. J.}, 16:377--383, 1940.

\bibitem{Pettis}
B.J. Pettis.
\newblock On continuity and openness of homomorphisms in topological groups.
\newblock {\em Ann. of Math. (2)}, 52:293--308, 1950.

\bibitem{Rosendal}
Christian Rosendal.
\newblock Automatic continuity of group homomorphisms.
\newblock {\em Bull. Symbolic Logic}, 15(2):184--214, 2009.

\bibitem{Rudin59}
Walter Rudin.
\newblock Positive definite sequences and absolutely monotonic functions.
\newblock {\em Duke Math. J}, 26:617--622, 1959.

\bibitem{Schoenberg42}
I.J. {S}choenberg.
\newblock Positive definite functions on spheres.
\newblock {\em Duke Math. J.}, 9:96--108, 1942.

\bibitem{Sierpinsky}
Wac{\l}aw Sierpi\'nsky.
\newblock Sur l'\'equation fonctionelle $f(x+y)=f(x)+f(y)$.
\newblock {\em Fund. Math.}, 1:116--122, 1920.

\bibitem{vasudeva79}
Harkrishan~L. Vasudeva.
\newblock Positive definite matrices and absolutely monotonic functions.
\newblock {\em Indian J. Pure Appl. Math.}, 10(7):854--858, 1979.

\bibitem{Wesolowski}
Jacek Weso{\l}owski.
\newblock Multiplicative {C}auchy functional equation and the equation of
  ratios on the {L}orentz cone.
\newblock {\em Studia Math.}, 179(3):263--275, 2007.

\end{thebibliography}

\end{document}